\documentclass[a4paper,11pt]{amsart}
\usepackage{amsmath,amssymb,amsfonts,amsthm}
\usepackage{latexsym,mathrsfs}
\usepackage{graphicx}
\usepackage[all]{xy}
\input xypic
\usepackage{color}
\usepackage{hyperref}
\hypersetup{colorlinks=true,linkcolor=red,citecolor=blue}
\usepackage[OT2,T1]{fontenc}
\usepackage[utf8]{inputenc}
\usepackage{lmodern}
\usepackage{microtype}
\usepackage{enumerate}
\usepackage{tikz-cd}
\usepackage{bm}
\theoremstyle{plain}
\newtheorem{theorem}[subsection]{{\bf Theorem}}
\newtheorem*{theorem*}{{\bf Theorem}}
\newtheorem{corollary}[subsection]{{\bf Corollary}}
\newtheorem*{corollary*}{{\bf Corollary}}
\newtheorem{proposition}[subsection]{{\bf Proposition}}
\newtheorem{lemma}[subsection]{{\bf Lemma}}

\theoremstyle{definition}

\theoremstyle{remark}

\newtheorem{example}[subsection]{{\it Example}}
\numberwithin{equation}{subsection}
\DeclareMathOperator{\SL}{SL}
\DeclareMathOperator{\UT}{UT}

\DeclareBoldMathCommand{\bbot}{\bot}

\DeclareSymbolFont{cyrletters}{OT2}{wncyr}{m}{n}
\DeclareMathSymbol{\Sha}{\mathalpha}{cyrletters}{"58}
\DeclareMathOperator{\pwc}{pwc}
\DeclareMathOperator{\pwh}{pwh}
\DeclareMathOperator{\cl}{cl}
\begin{document}
\title[Powerful class]{The powerful class of groups}
\author{Primo\v z Moravec}
\address{{
Faculty of  Mathematics and Physics, University of Ljubljana,
and Institute of Mathematics, Physics and Mechanics,
Slovenia}}
\email{primoz.moravec@fmf.uni-lj.si}
\thanks{ORCID: \texttt{https://orcid.org/0000-0001-8594-0699}. The author acknowledges the financial support from the Slovenian Research Agency, research core funding No. P1-0222, and projects No. J1-3004, N1-0217, J1-2453, J1-1691. He would also like to thank the referee for simplifying some of the arguments.}
\date{\today}
%%%%%%%%%%%%%%%%%%%%%%%%%%%%%%%%%%%%%%%%%%%%%%%%%%%%%%%%%%%%%%%%%%%%%
\begin{abstract}
\noindent
Pro-$p$ groups of finite powerful class are studied. We prove that these are $p$-adic analytic, and further describe their structure when their powerful class is small. It is also shown that there are only finitely many finite $p$-groups of fixed coclass and powerful class.
\end{abstract}
%%%%%%%%%%%%%%%%%%%%%%%%%%%%%%%%%%%%%%%%%%%%%%%%%%%%%%%%%%%%%%%%%%%%%%%
\maketitle
%%%%%%%%%%%%%%%%%%%%%%%%%%%%%%%%%%%%%%%%%%%%%%%%%%%%%%%%%%%%%%%%%%%%%%%
\section{Introduction}
\label{s:intro}

\noindent
Throughout this paper we assume that $p$ is an odd prime.

Powerful pro-$p$ groups play a fundamental role in Lazard's characterization of $p$-adic analytic groups \cite{Laz65}. In addition to that, their finite counterparts were first systematically discussed by Lubotzky and Mann \cite{LM87}, and they turn out to share several properties with abelian groups. 

Mann \cite{Man11} introduced the notion of powerful class of a finite $p$-group $G$ by considering ascending series of normal subgroups with consecutive quotients being powerfully embedded in the corresponding quotient of $G$. He demonstrated that finite $p$-groups of small powerful class have well behaved power structure, and thus they are not far away from being powerful.

The purpose of this note is to consider pro-$p$ groups of finite powerful class. These are common generalizations of powerful or nilpotent pro-$p$ groups. Our first main result goes as follows:

\begin{theorem*}
    Let $G$ be a finitely generated pro-$p$ group of finite powerful class. Then $G$ is $p$-adic analytic. The set of all elements of $G$ of finite order forms a finite subgroup of $G$.
\end{theorem*}

The second part of the above result follows from obtaining a bound for the exponent of $\Omega_i(G)$ in terms of $i$ and powerful class in the case when $G$ is a finite $p$-group. The argument relies on techniques developed by Fern\' andez-Alcober, Gonz\' alez-S\' anchez and Jaikin-Zapirain \cite{FGJ08}.

We proceed by looking into pro-$p$ groups of small powerful class. These are closely related to pro-$p$ groups admitting potent filtrations, also known as PF-groups:

\begin{theorem*}
    Every finitely generated pro-$p$ group of small powerful class is a PF-group.
\end{theorem*}

Gonz\'{a}lez-S\'{a}nchez \cite{GS07} showed that torsion-free PF-groups are precisely the $p$-saturable groups. These groups naturally admit a Lie algebra structure that  turns the group into a $p$-saturable Lie algebra. If $G$ is a finitely generated pro-$p$ group of small powerful class, the above result shows that $G$ is $p$-saturable. We show that the corresponding Lie algebra also has small powerful class. On the other hand, we exhibit an example showing that Kirillov's orbit method cannot be applied in general to derive the irreducible representations of a torsion-free pro-$p$ group of small powerful class.

If $G$ is a finite $p$-group of order $p^n$ and $c$ is its nilpotency class, then $n-c$ is called the {\it coclass} of $G$. Coclass theory \cite{LGM02} works towards understanding the structure of finite $p$-groups according to coclass. We show:

\begin{theorem*}
    Given $p$, $r$ and $k$, there are only finitely many finite $p$-groups of
coclass $r$ and powerful class at most $k$.
\end{theorem*}

The proof uses Shalev's detailed description of the uniserial structure of large finite $p$-groups of given coclass, {\it cf}. \cite{LGM02}.
A similar method shows that there are only finitely many PF $p$-groups of fixed coclass.

\section{Powerful class}
\label{s:pwc}

\noindent
A normal subgroup $N$ of a finite $p$-group $G$ is {\it powerfully embedded} in $G$ if $[N,G]\le N^p$. Similarly, if $G$ is a pro-$p$ group and $N$ a closed normal subgroup of $G$, then $N$ is powerfully embedded in $G$ if $[N,G]\le N^p$. Here $N^p$ stands for the closure of the abstract group $N^p$, we omit the closure operator throughout the text. It is easy to see that, in the pro-$p$ setting, $N$ is powerfully embedded in $G$ if and only $NK/K$ is powerfully embedded in $G/K$ for all open normal subgroups $K$ of $G$. If $G$ is powerfully embedded in itself, we say that $G$ is a {\it powerful group}. 
The definition is slightly different when $p=2$, as the condition of being powefully embedded is stated as  $[N,G]\le N^4$. But since we always assume that $p>2$, we will not use that. 

Let $G$ be a finite $p$-group. Denote by $\eta(G)$ the largest powefully embedded subgroup of $G$. Note that $\eta(G)$ is the product of all powerfully embedded subgroups of $G$. Clearly we have that $Z(G)$ is contained in $\eta (G)$. 

We recall the notion of powerful class introduced by Mann \cite{Man11}.
The {\it upper $\eta$-series of $G$} is defined by $\eta_0(G)=1$ and $$\eta_{i+1}(G)/\eta_i(G)=\eta(G/\eta_i(G))$$
for $i\ge 0$. The smallest $k$ with $\eta_k(G)=G$ is called the {\it powerful class} of $G$. We use the notation $\pwc(G)=k$. Ocasionally we use the shorthand notation $\eta_i$ for $\eta_i(G)$.
An ascending series
$1=N_0\le N_1\le N_2\le \cdots$
of normal subgroups of $G$ is said to be an {\it $\eta$-series} if $N_{i+1}/N_i$ is powerfully embedded in $G/N_i$ for all $i$. The shortest length $k$ of an $\eta$-series with $N_k=N$ is called the {\it powerful height} of $N$. It is denoted by $\pwh(N)$. A group $G$ is said to have {\it small powerful class} if $\pwc(G)<p$. Similarly, a normal subgroup $N$ of $G$ has {\it small powerful height} if $\pwh(N)<p$.

It is easily seen that the upper $\eta$-series is the fastest growing $\eta$-series in a group:

\begin{proposition}
    \label{p:etafast}
    Let $G$ be a finite $p$-group. Let $1=N_0\le N_1\le\cdots\le N_k=G$ be an $\eta$-series in $G$. Then $N_i\subseteq \eta_i(G)$.
\end{proposition}

\begin{proof}
    The claim is true for $i=0,1$. Suppose it holds for some $i\ge 1$. The group $N_{i+1}/{N_i}$ is powerfully embedded in $G/N_i$. By induction assumption, we have $N_i\subseteq \eta_i$. Therefore $N_{i+1}\eta_i/\eta_i$ is powerfully embedded in $G/\eta_i$. It follows from here that $N_{i+1}\eta_i/\eta_i\subseteq \eta (G/\eta_i)=\eta_{i+1}/\eta_i$. Hence we get $N_{i+1}\subseteq \eta_{i+1}$, as required.
\end{proof}

The notion of powerful class can be extended to the pro-$p$ setting. 
We say that a pro-$p$ group $G$ has {\it finite powerful class} if it has an $\eta$-series of closed subgroups of finite length that ends in $G$.
Given a pro-$p$ group $G$, define $\eta(G)$ be the product of all closed normal subgroups of $G$ that are powerfully embedded in $G$. Then $\eta (G)$ is a closed subgroup of $G$ containing all powerfully embedded subgroups of $G$. The upper $\eta$-series of $G$ can be defined as in the finite case. Then $G$ has finite powerful class if and only there exists $k$ such that $\eta_k(G)=G$. The smallest such $k$ is the {\it powerful class} of $G$. If a pro-$p$ group $G$ has powerful class $\le k$, then it is an inverse limit of finite $p$-groups of powerful class $\le k$.
%All the properties of the $\eta$-series derived in Lemma \ref{l:etaseries}, except for (2), still hold for pro-$p$ groups.

% Induction on $i$ shows that $\eta_i(G)K/K\le \eta_i(G/K)$ for every closed normal subgroup $K$ of $G$. Thus a pro-$p$ group $G$ has powerful class $\le k$ if and only if it is an inverse limit of finite $p$-groups of powerful class $\le k$.

We first collect some properties of the upper $\eta$-series. These will be used throughout the text without further reference. 

\begin{lemma}
    \label{l:center}
    let $G$ be a finitely generated pro-$p$ group. Then
    $$\eta(G)/\eta(G)^p=Z(G/\eta(G)^p).$$
\end{lemma}

\begin{proof}
The claim follows from a more general formula $$\eta_{k+1}/\eta_{k+1}^p\eta_k=Z(G/\eta_{k+1}^p\eta_k),$$ which holds for all $k\ge 0$. Namely, denote $N/\eta_{k+1}^p\eta_k=Z(G/\eta_{k+1}^p\eta_k)$. As $[\eta_{k+1},G]\le \eta_{k+1}^p\eta_k$ holds by definition, we have $\eta_{k+1}\subseteq N$. Conversely, $[N,G]\le \eta_{k+1}^p\eta_k\le N^p\eta_k$ shows that $N/\eta_k$ is powerfully embedded in $G/\eta_k$. Therefore, $N/\eta_k\le \eta(G/\eta_k)=\eta_{k+1}/\eta_k$, which concludes the proof.
    % Denote $Z(G/\eta(G)^p)=K/\eta (G)^p$. Since $\eta(G)/\eta (G)^p$ is central in $G/\eta(G)^p$, we conclude that $\eta(G)$ is contained in $K$. On the other hand, we have that
    % $[K,G]\le \eta(G)^p\le K^p$.
    % This shows that $K$ is powefully embedded in $G$. By definition, it follows that $K$ is contained in $\eta(G)$.
\end{proof}

\begin{lemma}
    \label{l:etaseries} 
    Let $G$ be a finitely generated pro-$p$ group. Then the following hold:
    \begin{enumerate}
        \item $Z_i(G)\le \eta_i(G)$ for all $i\ge 0$.
        \item If $G$ is a nilpotent group, then $\pwc(G)\le \cl (G)$, where $\cl(G)$ is the nilpotency class of $G$.
        \item $\eta_i(G/\eta_j(G))=\eta_{i+j}(G)/\eta_j(G)$.
        \item $\pwh (\eta_i(G))\le i$.
        \item $\pwc (\eta_i(G))\le i$.
        \item $[\eta_i(G),{}_iG]\le \eta_i(G)^p$.
        \item $\eta_i(G/\eta(G)^p)=\eta_i(G)/\eta(G)^p$ for all $i\ge 1$.
    \end{enumerate}
\end{lemma}

\begin{proof}
    Denote $Z_i=Z_i(G)$. The property (1) obviously holds for $i=0,1$. Suppose the assertion holds for some $i\ge 1$. As $[Z_{i+1},G]\le Z_i\le\eta_i$, it follows that $Z_{i+1}\eta_i/\eta_i$ is powerfully embedded in $G/\eta_i$. Thus $Z_{i+1}\eta_i/\eta_i\le\eta(G/\eta_i)=\eta_{i+1}/\eta_i$, and the assertion is proved for $i+1$ as well. In particular, (2) follows directly from here.

    We prove (3) by induction on $i$. We have the required equality for $i=0,1$, so we may assume it holds for some $i\ge 1$. Denote $N/\eta_j=\eta_{i+1}(G/\eta_j)$.
    It follows that $[N/\eta_j,G/\eta_j]\le (N/\eta_j)^p\eta_i(G/\eta_j)$. By induction assumption, this gives $[N,G]\le N^p\eta_{i+j}$. This implies that $N\eta_{i+j}/\eta_{i+j}$ is powerfully embedded in $G/\eta_{i+j}$, therefore
    $N\eta_{i+j}/\eta_{i+j}\le\eta(G/\eta_{i+j})=\eta_{i+j+1}/\eta_{i+j}$. We conclude that $N\le \eta_{i+j+1}$. Conversely, we have that $[\eta_{i+j+1},G]\le \eta_{i+j+1}^p\eta_{i+j}$ by definition. This can be restated as the fact that the quotient group $(\eta_{i+j+1}/\eta_j)/(\eta_i(G/\eta_j))$ is powerfully embedded in $(G/\eta_j)/\eta_i(G/\eta_j)$. Therefore we have that
    $\eta_{i+j+1}/\eta_j$ is contained in $\eta_{i+1}(G/\eta_j)=N/\eta_j$.

    (4) is obvious by definition. To prove (5), we use induction on $i$. We may assume that the inequality holds for $i\ge 1$ and for all groups $G$. Let $P=\eta(\eta_{i+1})$. Then we obviously have that $\eta(G)$ is contained in $P$. Therefore 
    \begin{align*}
        \pwc (\eta_{i+1}) &=\pwc (\eta_{i+1}/P)+1\\
        &\le \pwc (\eta_{i+1}/\eta_1)+1\\
        &=\pwc(\eta_i(G/\eta_1))+1\le i+1.
    \end{align*}

    Note that (6) holds for $i=0,1$. Assume it holds for $i\ge 1$. Then
    $[\eta_{i+1},{}_{i+1}G]\le [\eta_{i+1}^p\eta_i,{}_iG]=[\eta_{i+1}^p,{}_iG][\eta_i,{}_iG]\le \eta_{i+1}^p\eta_i^p=\eta_{i+1}^p$.

    Let us prove (7). Denote $N_i/\eta^p=\eta_i(G/\eta ^p)$ for $i\ge 1$, and $N_0=\eta^p$. We have that $[N_i,G]\le N_i^pN_{i-1}\eta^p=N_i^pN_{i-1}$ for all $i\ge 1$. 
    We proceed by induction.
    In the case when $i=1$, 
    notice that Lemma \ref{l:center} gives $\eta/\eta^p=Z(G/\eta^p)\le \eta(G/\eta^p)$, therefore $\eta\le N_1$.
    Conversely, 
    the fact that $[N_1,G]\le N_1^p$ shows that $N_1$ is contained in $\eta(G)$. 
    Asssume the claim holds for some $i\ge 1$. First, we have that $[N_{i+1},G]\le N_{i+1}^pN_i=N_{i+1}^p\eta_i$. Thus $N_{i+1}\eta_i/\eta_i$ is powerfully embedded in $G/\eta_i$. This shows that $N_{i+1}\eta_i/\eta_i$ is contained in $\eta(G/\eta_i)=\eta_{i+1}/\eta_i$, therefore $N_{i+1}\le\eta_{i+1}$. On the other hand, $[\eta_{i+1}/\eta^p,G/\eta^p]=[\eta_{i+1},G]\eta^p/\eta^p\le \eta_{i+1}^p\eta_i/\eta^p=(\eta_{i+1}/\eta^p)^p\eta_i(G/\eta^p)$. This demonstrates that $\eta_{i+1}/\eta^p$ is contained in $\eta_{i+1}(G/\eta^p)$, hence $\eta_{i+1}\le N_{i+1}$.
\end{proof}

The next lemma gives some information on pro-$p$ groups with powerful class two:

\begin{lemma}
    \label{l:pwc2}
    Let $G$ be a finitely generated pro-$p$ group and suppose that $G/\eta(G)$ is powerful. Then $G/\eta(G)$ is elementary abelian.
\end{lemma}

\begin{proof}
    Denote $Q=G/\eta(G)^p$.
    By Lemma \ref{l:center} we conclude that $Q/Z(Q)$ is powerful. From \cite[Proposition 3.6]{Wil21a} (the pro-$p$ version has the same proof) it follows that $Q^pZ(Q)$ is powerfully embedded in $Q$. We quickly deduce that $G^p\eta(G)$ is powerfully embedded in $G$, therefore $G^p\le \eta(G)$. The quotient $G/\eta(G)$ is thus a powerful group of exponent $p$, hence it is abelian.
\end{proof}

\begin{corollary}
    \label{c:k-1}
    Let $G$ be a finitely generated pro-$p$ group of powerful class $k$. Then $\eta_{k-1}(G)$ is open in $G$.
\end{corollary}

\begin{proof}
    Note that $G/\eta_{k-1}$ is powerful. This implies that $(G/\eta_{k-2})/\eta(G/\eta_{k-2})$ is powerful. By Lemma \ref{l:pwc2} we have that $\Phi(G/\eta_{k-2})\le \eta(G/\eta_{k-2})$. Thus $\Phi(G)\le \eta_{k-1}$, and this concludes the proof. 
\end{proof}

We are ready to prove the first half of our first main result mentioned in the introduction:

\begin{proposition}
    \label{p:padic}
    Let $G$ be a finitely generated pro-$p$ group of finite powerful class. Then $G$ is $p$-adic analytic.
\end{proposition}

\begin{proof}
    Let $\pwc (G)=k$. We prove the result by induction on $k$. Clearly, the result holds true for $k=0,1$. Assume it holds for groups of powerful class $\le k-1$. By Corollary \ref{c:k-1}, we have that $\eta_{k-1}$ is open in $G$, therefore it is finitely generated. As $\pwc(\eta_{k-1})\le k-1$, we have that $\eta_{k-1}$ is $p$-adic analytic. Therefore $G$ is $p$-adic analytic.
\end{proof}

It is straightforward to see that if $G$ is a finitely generated pro-$p$ group with an open powerfully embedded subgroup, then $G$ has finite powerful class. When $G$ is nilpotent, the converse also holds:

\begin{proposition}
    \label{p:nilpotent}
    Let $G$ be a finitely generated nilpotent pro-$p$ group. Then $\eta(G)$ is open in $G$.
\end{proposition}

\begin{proof}
    Note that $Z(G/\eta(G)^p)=\eta(G)/\eta(G)^p$ is a finite $p$-group, since $\eta(G)$ is finitely generated. As $G/\eta(G)^p$ is nilpotent, we get from here that it is finite \cite[5.2.22]{Rob82}. Thus the result follows.
\end{proof}

\begin{example}
    \label{ex:SLnZp}
    Let $n>2$ and let $S$ be a Sylow pro-$p$ subgroup of $\SL_n(\mathbb{Z}_p)$. Then $S$ can be seen as an inverse limit of upper unitriangular groups $\UT_n(\mathbb{Z}/p^m\mathbb{Z})$ \cite[p. 35]{DdSMS99}. Thus $S$ is nilpotent of class $n-1$. It follows that $\pwc(S)\le n-1$ and $\eta(S)$ is open in $S$. One can verify that $\eta(S)$ consists precisely of all those upper unitriangular matrices $(a_{ij})$ with $a_{i,i+\ell}\in p^{n-\ell-1}\mathbb{Z}_p$.
\end{example}

We end this section by mentioning the relationship with capability of groups. We say that a group $G$ is {\it capable} if there exists a group $Q$ with $Q/Z(Q)\cong G$. It is well known that non-trivial cyclic groups are not capable. Baer \cite{Bae38} classified finite abelian groups that are capable. We define a finite $p$-group $G$ to be {\it $\eta$-capable} if there exists a finite $p$-group $P$ with $P/\eta(P)\cong G$. Again, it is easy to see that a non-trivial cyclic group cannot be $\eta$-capable, see, for instance, \cite[p. 45]{DdSMS99}. Note that if a finite $p$-group $G$ is $\eta$-capable with $P/\eta(P)\cong G$, then $(P/\eta(P)^p)/Z(P/\eta(P)^p)\cong G$ by Lemma \ref{l:center}. This shows that $\eta$-capability implies the usual capability. The converse does not hold. The group $C_{p^2}\times C_{p^2}$ is capable by \cite{Bae38}, yet Lemma \ref{l:pwc2} shows that it is not $\eta$-capable, as all abelian $\eta$-capable $p$-groups are elementary abelian.

\section{Pro-$p$ groups of small powerful class}
\label{s:small}

\noindent
Recall that a pro-$p$ group $G$ is said to have small powerful class if $\pwc(G)<p$. Note that if a stronger condition $\pwc(G)<p-1$ holds, then $G$ satisfies the condition $\gamma_{p-1}(G)\le G^p$. Groups satisfying this property are called {\it potent} and are thoroughly described by Gonz\'{a}lez-S\'{a}nchez and Jaikin-Zapirain \cite{GJ04}. We are thus more or less only interested in the case $\pwc(G)=p-1$.

Mann's results on finite $p$-groups of small powerful class are summarized below. One may verify that similar properties hold for pro-$p$ groups of small powerful class and corresponding closed normal subgroups of small powerful height:

\begin{proposition}[\cite{Man11}]
    \label{p:mann}
    Let $G$ be a finite $p$-group.
    \begin{enumerate}
        \item If $G$ has small powerful class, then $G^p$ is powerful, and $G^p=\{ x^p\mid x\in G\}$.
        \item If $G=\langle a,b\rangle$ has small powerful class and $a^{p^e}=b^{p^e}=1$, then $G^{p^e}=1$.
        \item If $N$ is a normal subgroup of $G$ with small powerful height, then $[N^{p^k},G]=[N,G]^{p^k}\le [N,G^{p^k}]$.
    \end{enumerate}
\end{proposition}

Let $G$ be a pro-$p$ group. A closed normal subgroup $N$ of $G$ is {\it PF-embedded} in $G$ \cite{FGJ08} if there exists a $G$-central series $N=N_1\ge N_2\ge \cdots$  with trivial intersection $\cap_{i\in\mathbb N}N_i$ and $[N_i,{}_{p-1}G]\le N_{i+1}^p$ for all $i$. Such a series is called a {\it potent filtration} of $N$ in $G$. We also say that $G$ is a {\it PF-group} if it is PF-embedded in itself. It is clear that $N$ is PF-embedded in $G$ if and only if $NK/K$ is PF-embedded in $G/K$ for all open normal subgroups $K$ in $G$.

\begin{proposition}
    \label{p:smallPF}
    Let $G$ be a finitely generated pro-$p$-group and $N$ a normal subgroup of $G$. If $N$ has small powerful height, it is PF-embedded in $G$.
\end{proposition}

\begin{proof}
    In the course of the proof, we use Proposition \ref{p:mann} (3) without further explicit reference.
    Let $1=N_0\le N_1\le \cdots \le N_{p-1}=N$ be an $\eta$-series of $N$ in $G$. Denote $N_j=1$ for $j<0$ and $N_\ell=N$ for $\ell\ge p$. Define
    \begin{align*}
        M_1 &= N,\\
        M_{i+1} &= M_i^pN_{p-i-1}.
    \end{align*}
    Note that all $M_i$ have small powerful height \cite[Lemma 2.5]{Man11}. Induction shows that we have a descending series $N=M_1\ge M_2\ge\cdots$. Let us first prove that this is a central series. Note that $[M_1,G]=[N,G]=[N_{p-1},G]\le N_{p-1}^pN_{p-2}=M_2$. Suppose that we have $[M_i,G]\le M_{i+1}$. Then $[M_{i+1},G]=[M_i^pN_{p-i-1},G]=[M_i^p,G][N_{p-i-1},G]\le [M_i,G]^pN_{p-i-1}^pN_{p-i-2}=M_{i+2}$, as $N_{p-i-1}^p\le M_{i+1}^p$.

    Now we show that $[M_i,{}_{p-1}G]\le M_{i+1}^p$. This holds for $i=1$, as the fact that we have a central series implies that
    $[M_1,{}_{p-1}G]\le M_p=M_{p-1}^pN_{-1}=M_{p-1}^p\le M_2^p$. For induction step, we may assume that $M_{i+2}^p=1$. 
    From $[N_{p-i-1},G]\le N_{p-i-1}^pN_{p-i-2}$ we readily obtain 
    $[N_{p-i-1},G,G]\le [N_{p-i-1},G]^p[N_{p-i-2},G]\le (N_{p-i-1}^pN_{p-i-2})^pN_{p-i-2}^pN_{p-i-3}=N_{p-i-3}$.
    Induction on $k$ shows that, under the above assumption, we have $[N_{p-i-1},{}_kG]\le N_{p-i-k-1}$ for all $k\ge 2$. 
    Finally, note that the above implies $[M_{i+1},{}_{p-1}G]=[M_{i}^pN_{p-i-1},{}_{p-1}G]=
    [M_i,{}_{p-1}G]^p[N_{p-i-1},{}_{p-1}G]\le (M_{i+1}^p)^p[N_{p-i-1},{}_{p-1}G]=[N_{p-i-1},{}_{p-1}G]\le N_{-i}=1$, as required.

    Since $M_{p+k}=M_{p-1}^{p^{k+1}}$ for all $k\ge 0$, we quickly conclude that the intersection of all $M_i$ is trivial. This finishes the proof.
\end{proof}

%The proof of Proposition \ref{p:smallPF} can easily be adjusted to the pro-$p$ setting. It shows, in particular, the following:

\begin{corollary}
    \label{c:smallpwc}
    Every finitely generated pro-$p$ group of small powerful class is a PF-group.
\end{corollary}

Every torsion-free pro-$p$ group of small powerful class is therefore $p$-saturable in the sense of Lazard \cite{Laz65}. The latter have a natural $\mathbb{Z}_p$-lattice structure, first discovered by Lazard ({\it op. cit.}) and further developed by Gonz\'{a}lez-S\'{a}nchez \cite{GS07}. If $G$ is a $p$-saturable group, then the following operations turn it into a $p$-saturable Lie algebra $\mathcal{G}=G$:
\begin{align*}
    x+y &=\lim_{n\to\infty}(x^{p^n}y^{p^n})^{p^{-n}},\\
    \lambda x &= x^\lambda,\\
    [x,y]_{\rm Lie} &= \lim_{n\to\infty}[x^{p^n},y^{p^n}]^{p^{-2n}}.
\end{align*}
Conversely, every $p$-saturable Lie algebra becomes a $p$-saturable group with multiplication given via the Baker--Campbell--Hausdorff formula
$$\Phi(x,y)=\log(\exp x\cdot \exp y)=x+y+\sum_{i=2}^\infty u_i(x,y),$$
where $u_i(x,y)$ are Lie polynomials in $x$ and $y$ of degree $i$ with coefficients in $\mathbb{Q}$, see \cite[Theorem 6.28]{DdSMS99} for further details.

If $\mathcal{L}$ is a $\mathbb{Z}_p$-Lie algebra, then a subalgebra $\mathcal{K}$ is {\it powerfully embedded} in $\mathbb{L}$ if $[\mathcal{K},\mathcal{L}]_{\rm Lie}\le p\mathcal{K}$. Analogously, one extends the notion of PF-embedded subgroups to PF-embedded Lie subalgebras \cite{GS07}. Furthermore, we can define the powerful class for $\mathbb{Z}_p$-Lie algebras as follows. A series $0=\mathcal{L}_0\le\mathcal{L}_1\le\cdots$ of ideals of a $\mathbb{Z}_p$-Lie algebra $\mathcal{L}$ is an {\it $\eta$-series} if $\mathcal{L}_{i+1}/\mathcal{L}_i$ is powerfully embedded in $\mathcal{L}/\mathcal{L}_i$ for all $i$. If there is an $\eta$-series of $\mathcal{L}$ that reaches $\mathcal{L}$ in finitely many steps, we say that $\mathcal{L}$ has {\it finite powerful class}. In this case, the length of shortest $\eta$-series of $\mathcal{L}$ is called the {\it powerful class} $\pwc(\mathcal{L})$ of $\mathcal{L}$. Denote by $\eta(\mathcal{L})$ the sum of all powerfully embedded ideals in $\mathcal{L}$. Then we can define the upper $\eta$-series of a Lie algebra exactly the same as in the group case. It is also clear that the upper $\eta$-series is the fastest growing $\eta$-series of the Lie algebra $\mathcal{L}$. 

\begin{corollary}
    \label{c:smallLie}
    A finitely generated torsion-free pro-$p$ group $G$ has small powerful class if and only if the corresponding Lie algebra $\mathcal{G}$ has small powerful class. In this case, $\pwc (G)=\pwc(\mathcal{G})$ and $\eta_i(G)=\eta_i(\mathcal{G})$ for all $i\ge 0$.
\end{corollary}

\begin{proof}
    Suppose $G$ has small powerful class $k$. The group $G$ is $p$-saturable, therefore the corresponding Lie algebra $\mathcal{G}$ is $p$-saturable \cite[Theorem 4.2]{GS07}. Let $1=N_0\le N_1\le\cdots \le N_k=G$ be an $\eta$-series of $G$ with $k<p$. Then all subgroups $N_i$ are PF-embedded in $G$ by Proposition \ref{p:smallPF}. By
    \cite[Theorem 4.5]{GS07}, we have a corresponding series of PF-embedded ideals of $\mathcal{G}$ given as $0=\mathcal{N}_0\le\mathcal{N}_1\le\cdots\le \mathcal{N}_k=\mathcal{G}$, and
    $$[\mathcal{N}_{i+1},\mathcal{G}]_{\rm Lie}=[N_{i+1},G]\le N_{i+1}^pN_i=p\mathcal{N}_{i+1}+\mathcal{N}_i$$
    for all $i$. This shows that $(\mathcal{N}_i)_i$ is an $\eta$-series of $\mathcal{G}$, hence $\pwc(\mathcal{G})\le k$.

    The converse follows from the fact that if $(\mathcal{N}_i)_i$ is an $\eta$-series of $\mathcal{G}$, then an analogous argument as in the proof of Proposition \ref{p:smallPF} shows that all $\mathcal{N}_i$ are PF-embedded in $\mathcal{G}$. Then the argument proceeds along the similar lines as in the previous paragraph. 

    The equality of the upper $\eta$-series of $G$ and $\mathcal{G}$ now follows from \cite[Theorem 4.5]{GS07}.
\end{proof}

Kazhdan \cite{Kaz77} showed that Kirillov's orbit method 
provides a correspondence between the irreducible characters of finite $p$-groups of class $<p$ and the orbits of the action of that group on the dual space of the corresponding Lie algebra. 
In \cite{GS09}, Gonz\' alez-S\' anchez showed that the orbit method also works for some classes of $p$-saturable groups, such as torsion-free potent groups.
However, the orbit method no longer works for $p$-saturable groups of small powerful class:

\begin{example}[Example 1 of \cite{GS09}]
    \label{ex:kirillov}
    Let $M=\langle x_1,x_2\ldots ,x_p\rangle\cong\mathbb{Z}_p^p$ and form $G=\langle \alpha\rangle\ltimes M\cong \mathbb{Z}_p\ltimes\mathbb{Z}_p^p$, where the action of $\alpha$ on $M$ is given by $[x_i,\alpha]=x_{i+1}$ for $i\le p-2$, and $[x_{p-1}, \alpha]=x_p^p$ and $[x_p,\alpha]=1$. Then we readily get that
    $$\eta_i(G)=\langle \alpha^{p^{p-i-1}},x_1^{p^{k_{i1}}},x_2^{p^{k_{i2}}},\ldots ,x_{p-2}^{p^{k_{i,p-2}}},x_{p-1},x_p\rangle,$$
    where $k_{ij}=\max\{p-i-j,0\}$, hence $\pwc(G)=p-1$. The group $G$ therefore has small powerful class, yet the orbit method does not yield all of its irreducible representations \cite{GS09}.
\end{example}

Pro-$p$ groups whose powerful class is not small may not be PF-groups, as the following example shows:

\begin{example}
    \label{ex:pwcnotPF}
    We exhibit a finite $p$-group of powerful class equal to $p$ that is not a PF-group.
    Let $M$ be an elementary abelian $p$-group with generators $x_1,x_2,\ldots ,x_p$. Form $G=\langle \alpha\rangle \ltimes M$, where $\alpha$ has order $p^2$ and acts on $M$ as follows: $[x_i,\alpha]=x_{i+1}$ for $i=1,2,\ldots ,p-1$, and $[x_p,\alpha]=1$. The group $G$ has order $p^{p+2}$ and nilpotency class $p$. Note that
    $$(\alpha x_1)^p=\alpha^px_1^px_2^{{p\choose 2}}x_3^{{p\choose 3}}\cdots x_p^{{p\choose p}}=\alpha^px_p,$$
    therefore $x_p\in G^p$. On the other hand, $x_p$ is not a $p$-th power of some element of $G$. This shows that $G$ is not a PF-group by \cite[Theorem 3.4]{FGJ08}. By Corollary \ref{c:smallpwc} we must have that $\pwc (G)=p$.
\end{example}

On the other hand, there are PF-groups, even torsion-free and potent, which do not have finite powerful class:

\begin{example}
    \label{ex:potentnopwc}
    In the following we construct a finitely generated torsion-free potent pro-$p$ group $G$, which does not have finite powerful class. Let $p>3$ and let $n$ be a positive integer. Let $G_n=\langle \alpha\rangle\ltimes M$, where $M=\langle x_1,x_2,\ldots ,x_{p-2}\rangle$ is an abelian group, and $|x_1|=p^{n+1}$ and $|x_2|=|x_3|=\cdots =|x_{p-2}|=|\alpha|=p^n$. The action of $\alpha$ on $M$ is given by $[x_i,\alpha]=x_{i+1}$ for $i=1,2,\ldots ,p-3$, and $[x_{p-2},\alpha]=x_1^p$. We clearly have that $\gamma_{p-1}(G_n)\le G_n^p$, hence $G_n$ is a two-generator potent $p$-group. From the relations it follows that $Z(G_n)=\langle x_1^{p^n}\rangle$. One can easily verify that this is the largest powerfully embedded subgroup of $G_n$, therefore $\eta(G_n)=Z(G_n)$. By taking successive quotients, one can see that $\eta_i(G_n)=Z_i(G_n)$ for all $i\ge 1$. As the nilpotency class of $G_n$ is precisely $n(p-2)+1$, we conclude that $\pwc G_n=n(p-2)+1$.

    The groups $G_n$ clearly form an inverse system. Their inverse limit $G\cong \mathbb{Z}_p\ltimes \mathbb{Z}_p^{p-2}$ is topologically generated by two generators, it is torsion-free and potent. As $\pwc (G_n)$ are not bounded, the group $G$ does not have finite powerful class.
\end{example}

\section{Elements of finite order and powerful class}
\label{s:orderfin}

\noindent
In this section we look at the elements of finite order in pro-$p$ groups of finite powerful class $k$. For $i\ge 0$, denote
$\Omega_i(G)=\langle x\in G\mid x^{p^i}=1\rangle$.  At first we bound the exponent of $\Omega_i(G)$ in terms of $i$ and $k$:

\begin{theorem}
    \label{t:omegai}
    Let $G$ be a finitely generated pro-$p$ group of powerful class $k$. Suppose $k\le \ell (p-1)$. Then $\Omega_i(G)^{p^{i+\ell}}=1$.
\end{theorem}

\begin{proof}
    We may assume that $G$ is finite.
    We prove by induction on $\ell$ that $\gamma_{\ell(p-1)}(G)$ is contained in some PF-embedded subgroup of $G$.  If $\ell =1$, then $G$ has small powerful class, therefore it is a PF-group by Corollary \ref{c:smallpwc}. For induction step, note that $G/\eta_{p-1}$ has powerful class $\le (\ell -1)(p-1)$. Therefore there exists a normal subgroup $N$ of $G$ such that $N/\eta_{p-1}$ is PF-embedded in $G/\eta_{p-1}$ and $\gamma_{(\ell-1)(p-1)}(G)\le N$. Choose a potent filtration
    $N/\eta_{p-1}=N_1/\eta_{p-1}\ge N_2/\eta_{p-1}\ge\cdots\ge N_r/\eta_{p-1}=1$ of $N/\eta_{p-1}$ in $G/\eta_{p-1}$. Then we have $[N_i,G]\le N_{i+1}$ and $[N_i,{}_{p-1}G]\le N_{i+1}^p\eta_{p-1}$ for all $i\ge 1$. The group $\eta_{p-1}$ is PF-embedded in $G$ by Proposition \ref{p:smallPF}. Choose a potent filtration $\eta_{p-1}=M_1\ge M_2\ge\cdots\ge M_s=1$ of $\eta_{p-1}$ in $G$. We claim that
    $$N^p\eta_{p-1}=N_1^p\eta_{p-1}\ge N_2^p\eta_{p-1}\ge\cdots\ge N_r^p\eta_{p-1}=\eta_{p-1}\ge M_2\ge\cdots\ge M_s=1$$
    is a potent filtration of $N^p\eta_{p-1}$ in $G$. The series is central, since
    $[N_i^p\eta_{p-1},G]\le [N_i,G]^p[N_i,{}_pG]\eta_{p-1}\le N_{i+1}^p\eta_{p-1}$. The proof will be concluded once we have shown that $[N_i^p\eta_{p-1},{}_{p-1}G]\le (N_{i+1}^p\eta_{p-1})^p$. To this end, we may assume that $(N_{i+1}^p\eta_{p-1})^p=1$ and $[N_i^p\eta_{p-1},{}_{j}G]=1$ for all $j\ge p$. At first note that $[\eta_{p-1},{}_{p-1}G]\le \eta_{p-1}^p=1$. We thus have that
    $[N_i^p\eta_{p-1},{}_{p-1}G]=[N_i^p,{}_{p-1}G]=[N_i,{}_{p-1}G]^p\le (N_{i+1}^p\eta_{p-1})^p=1$. This proves the claim.  We have therefore shown that $N^p\eta_{p-1}$ is PF-embedded in $G$. Observe that
    $\gamma_{\ell(p-1)}(G)\le [N,{}_{p-1}G]\le N_2^p\eta_{p-1}\le N^p\eta_{p-1}$.

    Our result now directly follows from \cite[Theorem 4.1]{FGJ08}.
\end{proof}

We note here that Mann \cite{Man11} constucted a finite $p$-group $G$ of small powerful class with $\exp\Omega_1(G)>p$, therefore the bound given in Theorem \ref{t:omegai} is close to being sharp. The bound can also be compared with Eeasterfield's bound for the exponent of $\Omega_i(G)$ in terms of $p$, $i$ and the nilpotency class of the group $G$, {\it cf}. \cite{FGJ08}.

An immediate consequence is the following:

\begin{corollary}
    \label{l:torsion}
    Let $G$ be a finitely generated pro-$p$ group of finite powerful class. Then the set of all torsion elements of $G$ forms a finite subgroup of $G$.
\end{corollary}

\section{Powerful class and coclass}
\label{s:coclass}

\noindent
If $G$ is a finite $p$-group of class $c$ and order $p^n$, then $c<n$. The number $r=n-c$ is called the {\it coclass} of $G$. Determining the structure of finite $p$-groups according to coclass has been very fruitful. We refer to \cite{LGM02}.

One of the important features of large $p$-groups of given coclass is that they act uniserially on certain parts of their lower central series by conjugation. Recall that a finite $p$-group $G$ acts {\it uniserially} on a finite $p$-group $N$ if $|H:[H,G]|=p$ for every non-trivial $G$-invariant subgroup $H$ of $N$. The following result due to Shalev is one of the fundamental results of the coclass theory:

\begin{lemma}[\cite{LGM02}, Theorem 6.3.9]
    \label{l:coclass}
Suppose $p>2$.
Let $G$ be a finite $p$-group of coclass $r$ and $|G|=p^n\ge p^{2p^r+r}$. Let $m=p^r-p^{r-1}$. Then there exists $0\le s\le r-1$ such that $G$ acts uniserially on $\gamma_m(G)$, and $\gamma_i(G)^p=\gamma_{i+d}(G)$ for all $i\ge m$, where $d=(p-1)p^s$. 
\end{lemma}

\begin{theorem}
    \label{t:pwccoclass}
    Given $p$, $r$ and $k$, there are only finitely many finite $p$-groups of coclass $r$ and powerful class at most $k$.
\end{theorem}

\begin{proof}
    Let $G$ be a finite $p$-group of coclass $r$ and powerful class $k$. Denote $|G|=p^n$ and suppose without loss of generality that $n\ge 2p^r+r$. The nilpotency class of $G$ is equal to $c=n-r\ge 2p^r$. Let $m$ and $d$ be as in Lemma \ref{l:coclass}. We have that $\pwh \gamma_m(G)\le k$ by \cite[Lemma 2.5]{Man11}. Consider an $\eta$-series 
    $1=N_0\lneq N_1\lneq\cdots \lneq N_\ell=\gamma_m(G)$ of $\gamma_m(G)$ in $G$, where
    $\ell\le k$. As $G$ acts uniserially on $\gamma_m(G)$, we have $N_i=\gamma_{m_i}(G)$ for some $m_i\ge m$ with
    $c+1=m_0\gneq m_1\gneq\cdots\gneq m_\ell=m$, see \cite[Lemma 4.1.3]{LGM02}. We thus have that $\gamma_{m_{i+1}}(G)/\gamma_{m_i}(G)$ is powerfully embedded in $G/\gamma_{m_i}(G)$ for all $i$. Therefore we see that $\gamma_{1+m_{i+1}}(G)\le \gamma_{m_{i+1}}(G)^p\gamma_{m_i}(G)=\gamma_{d+m_{i+1}}(G)\gamma_{m_i}(G)$ holds for all $i$. 
  Since $d>1$ and $m_i>m_{i+1}$, this is possible only if $m_i=m_{i+1}+1$. Then the above $\eta$-series of $\gamma_m(G)$ is uniserial, and we thus have that $|\gamma_m(G)|\le p^{k}$. On the other hand, $G/\gamma_m(G)$ has coclass $\le r$ and class $\le m-1$, thus $|G:\gamma_m(G)|\le p^{r+m-1}$. We conclude that $|G|\le p^{k+r+m-1}$, and this finishes the proof.
\end{proof}

\begin{corollary}
    \label{c:infpwccoclass}
    There is no infinite pro-$p$ group of finite coclass and finite powerful class.
\end{corollary}

We mention here an independent result that can be proved along similar lines:

\begin{proposition}
    \label{p“:PFcoclass}
    Given $p$ and $r$, there are only finitely many finite $p$-groups of coclass $r$ that are PF-groups.
\end{proposition}

\begin{proof}
    The proof follows along similar lines like the one of Theorem \ref{t:pwccoclass}. 
    Let $G$ be a PF-group of order $p^n$ and coclass $r$.
    Again, assume $n\ge 2p^r+r$, denote $c=n-r\ge 2p^r$, and let $m$ and $d$ be as in Lemma \ref{l:coclass}. The group $\gamma_m(G)$ is PF-embedded in $G$, see \cite[Proposition 3.2]{FGJ08}. As $G$ acts uniserially on $\gamma_m(G)$, there is a potent filtration of $\gamma_m(G)$ in $G$ that has the form $\gamma_m(G)=\gamma_{m_1}(G)\gneq \gamma_{m_2}(G)\gneq\cdots\gneq \gamma_{c+1}(G)=1$ for $m=m_1\lneq m_2\lneq\cdots$. By definition, $\gamma_{m+p-1}(G)\le \gamma_{m_2}(G)^p=\gamma_{m_2+d}(G)$. This gives that $m_2-m\le p-1-d\le 0$, a contradiction. Thus $G$ is not a PF-group.
\end{proof}

\begin{corollary}
    \label{c:infPFcoclass}
    There is no infinite pro-$p$ group of finite coclass that is also a PF-group.
\end{corollary}

A finite $p$ group of order $p^n$ and nilpotency class equal to $n-1$, where $n\ge 4$, is said to be of {\it maximal class}. We find here the upper $\eta$-series of finite $p$-groups of maximal class. At first we state the following:

\begin{proposition}
    \label{l:p3}
    Let $G$ be a nonabelian group of order $p^3$.
    \begin{enumerate}
        \item If $\exp G=p$, then $\eta(G)=Z(G)$.
        \item If $\exp G=p^2$, then $G$ is powerful and thus $\eta (G)=G$.
    \end{enumerate}
\end{proposition}

\begin{proposition}
    \label{p:maxclass}
    Let $G$ be a finite $p$-group of maximal class. Then $\eta(G)=Z(G)$.
\end{proposition}

\begin{proof}
    Denote $|G|=p^n$ and $|G:\eta(G)|=p^r$. Note that $r\ge 2$. By \cite[Proposition 3.1.2]{LGM02}, we have that $\eta (G)=\gamma_r(G)$. The subgroups $\gamma_k(G)$ are then powerfully embedded in $G$ for all $k\ge r$. As $Z(G)=\gamma_{n-1}(G)$, we also have that $r\le n-1$.

    Suppose that $r<n-1$. If $n\le p+1$, then both $G/\gamma_{n-1}(G)$ and $\gamma_2(G)$ have exponent $p$ \cite[Proposition 3.3.2]{LGM02}. It follows that $1\neq \gamma_{r+1}(G)=[\gamma_r(G),G]\le \gamma_r(G)^p=1$, a contradiction. Thus $n>p+1$. Suppose further that $r>n-p+1$. Then \cite[Corollary 3.3.6]{LGM02} yields that
    $1\neq \gamma_{r+1}(G)\le\gamma_r(G)^p=\gamma_{r+p-1}(G)$, which cannot happen. We conclude that $n-p+1<r<n-1$. Then $\gamma_r(G)^p\le \gamma_{n-p+1}(G)^p=\gamma_n(G)=1$, which implies $\gamma_{r+1}(G)=1$. This is a contradiction. We thus have $r=n-1$, hence the result.
\end{proof}

\begin{corollary}
    \label{c:etamaxclass}
    Let $G$ be a finite $p$-group of maximal class. Then $\eta_i(G)=Z_i(G)$ for all $i\ge 0$.
\end{corollary}

\begin{proof}
    Let $|G|=p^n$. The upper central series 
    $$1=Z_0(G)\le Z_1(G)\le\cdots \le Z_{n-1}(G)=G$$
    has all sections, except for the last one, of order $p$. Using Proposition \ref{p:maxclass} and induction, we conclude that $\eta_i(G)=Z_{i}(G)$ for $i\le n-3$. The group $G/\eta_{n-3}(G)$ has order $p^3$. Using the notations of \cite[p. 56]{LGM02}, we have that $G=\langle s,s_1\rangle$. Combination of Proposition 3.3.2, Proposition 3.3.3 and Lemma 3.3.7 of \cite{LGM02} shows that $s^p$ and $s_1^p$ both belong to $\gamma_3(G)=\eta_{n-3}(G)$. As $G/\gamma_3(G)$ has class 2 and $p>2$, it follows that $G^p\le \eta_{n-3}(G)$. Proposition \ref{l:p3} now shows that $\eta_{n-2}(G)/\eta_{n-3}(G)=\eta(G/\eta_{n-3}(G))=Z(G/Z_{n-3}(G))=Z_{n-2}(G)/\eta_{n-3}(G)$. Therefore $\eta_{n-2}(G)=Z_{n-2}(G)$ and $\eta_{n-1}(G)=Z_{n-1}(G)=G$.
\end{proof}

Therefore, if $G$ is a finite $p$-group of coclass $1$, then $\pwc(G)$ is equal to the nilpotency class of $G$. On the other hand, there are several $p$-groups of coclass two with powerful class strictly smaller than the nilpotency class. For example, there are four powerful $p$-groups of order $p^4$ and nilpotency class equal to $2$.

\end{document}